\documentclass[12pt]{amsart}
\usepackage[american]{babel}
\usepackage[utf8x]{inputenc} 
\usepackage{amsfonts, latexsym, amssymb, amsgen, amsbsy, amstext, amsopn, amsmath, txfonts, mathrsfs,yfonts}
\usepackage[sans]{dsfont}
\usepackage{cmll}
\usepackage{amsthm}
\usepackage{color}
\usepackage[all]{xy}
\usepackage{fancyhdr}
\usepackage{newlfont}
\usepackage{setspace}
\usepackage{verbatim}
\usepackage{hyperref} 

\usepackage{xcolor}

\AtBeginDocument{\def\MR#1{}} 

\usepackage[a4paper,top=3.25cm,bottom=3.2cm,left=26mm,right=26mm]{geometry}

\usepackage[T1]{fontenc}
\usepackage{amsthm}
\usepackage{color}

\renewcommand{\H}{\mathbb{H}}

\newcommand{\N}{\mathbb{N}}

\newcommand{\R}{\mathbb{R}}

\newcommand{\vH}{\varmathbb{H}}

\newcommand{\cp}{\mathfrak{p}}
\newcommand{\Hla}{\mathds{M}}
\newcommand{\cn}{\mathfrak{n}}
\newcommand{\kI}{\mathfrak{I}}

\newcommand{\Ht}{\mathds{M}}
\newcommand{\Ha}{\mathds{V}}
\newcommand{\Hb}{\mathds{W}}
\newcommand{\Pds}{\mathds{P}}

\newcommand{\beq}{\begin{equation}}
\newcommand{\eeq}{\end{equation}}

\newcommand{\lan}{\langle}
\newcommand{\ran}{\rangle}

\newcommand{\pa}[1]{\left( #1 \right)}               
\newcommand{\set}[1]{\left\{ #1 \right\}}            
\newcommand{\ban}[1]{\left\langle  #1 \right\rangle}  

\theoremstyle{plain}
\newtheorem{theorem}{Theorem}[section]
\newtheorem{lemma}[theorem]{Lemma}
\newtheorem{proposition}[theorem]{Proposition}
\newtheorem{corollary}[theorem]{Corollary}
\newtheorem{remark}[theorem]{Remark}

\newtheorem{example}[theorem]{Example}

\newcommand{\norm}[1]{\left\lVert#1\right\rVert}
\DeclareMathOperator{\Lie}{Lie}
\DeclareMathOperator{\ad}{ad}
\DeclareMathOperator{\spn}{span}
\DeclareMathOperator{\End}{End}
\DeclareMathOperator{\Id}{Id}
\DeclareMathOperator{\im}{Im}

\begin{document}
	
\title[The Michor--Mumford conjecture in Hilbertian H-type groups]{The Michor--Mumford conjecture in Hilbertian H-type groups}
\author{Valentino Magnani}
\address{Valentino Magnani, Dip.to di Matematica, Universit\`a di Pisa \\
	Largo Bruno Pontecorvo 5 \\ I-56127, Pisa}
\email{valentino.magnani@unipi.it}
\author{Daniele Tiberio}
\address{Daniele Tiberio, 
	SISSA, via Bonomea 265, 34136 Trieste, Italy}
\email{dtiberio@sissa.it}


\begin{abstract}
We introduce infinite dimensional Hilbertian H-type groups equipped with weak, graded, left invariant Riemannian metrics. 
For these Lie groups, we show that the vanishing of the geodesic distance and the local unboundedness of the sectional curvature coexist.  
The result validates a deep phenomenon conjectured in an influential 2005 paper by Michor and Mumford, namely, the vanishing of the geodesic distance is linked to the local unboundedness of the sectional curvature.
We prove that degenerate geodesic distances appear for a large class of weak, left invariant Riemannian metrics.
Their vanishing is rather surprisingly related to the infinite dimensional sub-Riemannian structure of Hilbertian H-type groups. The same class of weak Riemannian metrics yields the nonexistence of the Levi-Civita covariant derivative.
\end{abstract}

\thanks{V.M. acknowledges the APRISE - {\em Analysis and Probability in Science} project, funded by the University of Pisa, grant PRA 2022 85, and the MIUR Excellence Department Project awarded to the Department of Mathematics, University of Pisa, CUP I57G22000700001. D.T. was supported by the European Research Council (ERC) under the European Union's Horizon 2020 research and innovation programme (grant agreement No. 945655).
}

\maketitle

\tableofcontents


\section{Introduction}

The study of infinite dimensional manifolds is a vast and fascinating area
of Mathematics which, for instance, embraces Differential Geometry, Lie group theory, and PDEs.
A recent introduction to these topics can be found in the monograph \cite{Schmeding2023IDG} and the lecture notes \cite{Bruveris2016Notes,Maor2022Notes}.
Some foundational works are \cite{Hamilton1982,KriMic1997}, 
and more specific information on infinite dimensional Lie groups can be 
found for instance in \cite{Omori1979,KheWen2009GIDG} and in the surveys \cite{Milnor1984,Nee06:Tow,Glo06}.

We begin by highlightening Riemannian Hilbert manifolds, that constitute a class of infinite dimensional manifolds 
modeled on a Hilbert space. Their Riemannian metrics  induce the manifold topology on tangent spaces,
hence they are called {\em strong Riemannian metrics}
and their geodesic distance is a distance function, that separates points. While the local geometry of Riemannian Hilbert manifolds shares some analogies with the finite dimensional case, for instance 
the Levi-Civita connection always exists, certain global properties fail to hold.
For more information, we refer the reader to the interesting survey \cite{BilMer2017} and the references therein. 

Only for infinite dimensional manifolds, the so-called {\em weak Riemannian metrics} may induce on tangent spaces weaker topologies than the manifold topology, \cite{AMR88,Bruveris2016Notes, Maor2022Notes,Schmeding2023IDG}.
If the manifold is not modeled on a Hilbert space, then the Riemannian metric must be weak. 
We focus our attention on Lie groups modeled on a Hilbert space, where both weak and strong Riemannian metrics can be defined. Therefore, we use the terminology ``{\em strictly weak Riemannian metric}'' to emphasize the cases where the weak Riemannian metric is not strong.

For strictly weak Riemannian metrics, a striking phenomenon can occur, resulting in the vanishing of the geodesic distance between distinct points. We say that such geodesic distance is {\em degenerate}.
First examples of such surprising fact were discovered in different classes of Fréchet manifolds \cite{EP1993}, \cite{MM05}, \cite{MM06}. Simple examples of degenerate geodesic distances are also available in Hilbert manifolds, \cite{MagTib2020,MagTib2023}, see also 
\cite{BHP20,BauHesMao2023arxiv} for more information. 

In the 2005 paper \cite{MM05}, P.~Michor and D.~Mumford conjectured a relationship between the vanishing of the
geodesic distance and the local unboundedness of the sectional curvature. 
They proposed a fascinating interpretation behind this phenomenon: some parts of the 
infinite dimensional manifold ``wrap up on theirselves'' allowing to find curves connecting two distinct points and having shorter 
and shorter length, up to reaching vanishing infimum, see \cite{MM05} and \cite[Section~1.2]{BHP20}.
In Mathematical terms, we may rephrase this phenomenon as follows: whenever a weak Riemannian metric admits a degenerate geodesic distance, then the sectional curvature must be unbounded on some special sequences of planes that stay in a neighborhood of some point.
For infinite dimensional Lie groups, the homogeneity by translations allows this point to be the unit element.

We point out that when a specific choice of strictly weak Riemannian metric on the 
infinite dimensional Heisenberg group is considered, then the following phenomenon appears: the blow-up of the sectional curvature
occurs along some planes that are related to the shrinking curves which connect the distinct points, where the 
geodesic distance vanishes, \cite{MagTib2023}.
If we take into account the above comments and the subsequent Theorem~\ref{th:MichorMumford}, then we may interpret the Michor--Mumford conjecture in infinite dimensional Lie groups as follows.
Considering an infinite dimensional Lie group,
equipped with a weak, left invariant Riemannian metric and a degenerate geodesic distance, {\em then we expect that the sectional curvature at the unit element is positively unbounded}.

In the present paper, we introduce {\em Hilbertian H-type groups}, whose geometry validates the previous version of the conjecture, with respect to a large class of strictly weak, left invariant Riemannian metrics.
Hilbertian H-type groups include the infinite dimensional Heisenberg group of \cite{MagTib2023} and in the finite dimensional case
they exactly coincide with the well known H-type groups, that were discovered by A.\ Kaplan, \cite{Kaplan1980,Kaplan1981,Kaplan1983},
see also \cite{KapRic1982}.
We notice that Kaplan's definition perfectly works also through the infinite dimensional interpretation. 
On the other hand, the effective existence of infinite dimensional H-type groups needs to be verified. In Section~\ref{sect:Htype}, we provide an infinite dimensional construction, from which one may notice that there are infinitely many infinite dimensional Hilbertian H-type Lie groups, see Remark~\ref{rem:infiniteHtype}. 

We focus our attention on the ``natural'' {\em weak} Riemannian metrics on Hilbertian H-type groups, that are left invariant and make the subspaces $\Ha$ and $\Hb$ orthogonal. Borrowing the terminology from the finite dimensional case, we say that such metrics are {\em graded}. 
For instance, the Cameron-Martin subgroup of \cite{DriGor2008} is a two step, infinite dimensional Lie group  equipped with a strong and graded Riemannian metric. Thus, the next statement validates our interpretation of the Michor--Mumford conjecture in Hilbertian H-type groups. 

\begin{theorem}\label{th:MichorMumford}
Let $\sigma$ be a weak, graded Riemannian metric on a Hilbertian H-type group. If the metric $\sigma$ yields a degenerate geodesic distance, then the sectional curvature at the unit element exists on a sequence of planes and it is positively unbounded.
\end{theorem}
The starting point of the proof is that the degenerate geodesic distance forces the graded Riemannian metric to be strictly weak.
Then we prove that for strictly weak, graded Riemannian metrics the blow-up of the sectional curvature always occurs.
Extending Theorem~\ref{th:MichorMumford} to more general classes of infinite dimensional Lie groups seems an interesting open question.

It is also important to understand whether, and in which cases, the geodesic distance in a Hilbertian H-type gruoup is actually degenerate. Here a rather striking fact appears, since infinite dimensional sub-Riemannian (sub-Finsler) Geometry
enters the proof of Theorem~\ref{th:MichorMumford}. 
More generally, for {\em any} strictly weak, left invariant sub-Finsler metric on
a Hilbertian H-type group, the sub-Finsler distance is {\em degenerate}, see Theorem~\ref{thm:vanishingsubF}.
The idea behind the proof of this theorem is to use
a sequence of vectors, where the weak and the strong topology differ. Then we use the map $J$ associated with the structure of H-type group, which allows for the same ``shrinking-space effect'' that was first observed in \cite{MagTib2023}.
As a consequence, we have the following result, corresponding to Theorem~\ref{thm:vanishingFinslerdistances}.

\begin{theorem}[Characterization of points with vanishing distance]\label{thmintro:vanishingFinslerdistances}
Let $F$ be a strictly weak, left invariant Finsler metric on a Hilbertian H-type group $\Ht = \Ha \oplus \Hb$ and
let us fix $x,y \in \Ha$, $z_1 , z_2 \in \Hb$.
Then we have 
\[
d_F(x+z_1,y+z_2)=0\quad \text{if and only if}\quad x=y,
\]
where $d_F$ is the Finsler distance associated with $F$. 
\end{theorem}

The subclass of strictly weak, graded Riemannian metrics on a Hilbertian H-type group gives rise to another singular phenomenon, i.e. the lack of the Levi-Civita covariant derivative.

\begin{theorem}\label{thmintro:LeviCivita}
	If $\sigma$ is a strictly weak, graded Riemannian metric on a Hilbertian H-type group, then 
it does not admit the Levi-Civita covariant derivative.
\end{theorem}

The proof of the previous theorem is given in Section~\ref{sec:nonexLC}.
An example of nonexistence of the Levi-Civita connection was provided in \cite{BBM2014}. We also notice that in \cite[Example~2.26]{Maor2022Notes} the model of \cite{MagTib2020} is extended to a family of weak Riemannian metrics which do not possess Christoffel symbols, hence their associated 
Levi-Civita covariant derivative cannot exist.

Despite Theorem~\ref{thmintro:LeviCivita}, we observe that through the classical Arnold's formula \cite{Arn66} it is still possible to compute the sectional curvature of 
some planes in a Hilbertian H-type group. 
However, we cannot claim that the formula works for all planes.
In fact, there are also planes for which the Arnold's formula does not apply, as it is shown for instance in \cite[Remark~4.1]{MagTib2020}. 

The next theorem proves that in a Hilbertian H-type group equipped with a strictly weak, graded Riemannian metric, we can always find sequences of planes in $T_0\Ht$ where the sectional curvature is well defined, and also unbounded.

\begin{theorem}[Unboundedness of the sectional curvature]\label{thmintro:curvatures}
Let $\Ht$ be a Hilbertian H-type group.
If $\sigma$ is a strictly weak, graded Riemannian metric on $\Ht$, then there exist two sequences of planes 
$\{P_n\}_{n \in \N} , \{Q_n\}_{n \in \N} \subset T_0 \Ht$ 
whose sectional curvatures $K_{\sigma}(P_n)$ and $K_{\sigma}(Q_n)$ are well defined through the Arnold's formula and we have 
	$$
	\lim_{n \to \infty} K_{\sigma} (P_n) = - \infty
	\quad \text{   and   } \quad 
	\lim_{n \to \infty} K_{\sigma} (Q_n) = + \infty.
	$$
\end{theorem}
This theorem is a version of Theorem~\ref{thm:unboundedcurvature}, where
we provide an explicit form for the planes: 
\[
P_n=\spn\big\{w_n,J_zw_n\big\} \quad \text{and}\quad 
Q_n=\spn\big\{z,J_{Az}w_n\big\}.
\]
We believe that there is a connection between these planes and the sequence of curves that progressively decrease their length, while connecting the fixed distinct points. 
In fact, we point out that the projection on $\Ha$
of these {\em horizontal curves} has the form
\[
\gamma_1^n(t)= tc\sqrt{n}\,w_n+\frac{t^2c}{2}\frac{1}{\sqrt{n}}\,J_z(w_n)
\]
for $c\in\R$, see the proof of 
Theorem~\ref{thm:unboundedcurvature}.
In this sense, we surmise that the planes $P_n$ and $Q_n$ should be somehow related to the parts of the space where the curves $\gamma^n$ "move", when their length reduces up to converging to zero. However, the precise relationship between the planes of the blow-up and the shrinking curves remains unclear to us. 

We observe that Theorem~\ref{thmintro:curvatures} immediately gives Theorem~\ref{th:MichorMumford},
since the vanishing of the geodesic distance implies that the graded Riemannian metric is strictly weak. We hopefully expect that our remarks in Hilbertian H-type groups provide more insights to understand the Michor--Mumford phenomenon in other classes of infinite dimensional manifolds.

\section{Infinite dimensional H-type groups}\label{sect:Htype}

The approach of \cite{MagRaj2014} can be used to construct specific classes of infinite dimensional Banach nilpotent Lie groups, starting from an infinite dimensional nilpotent Lie algebra. Indeed, the group operation is immediately provided by the Baker--Campbell--Hausdorff formula,
which we abbreviate as ``BCH formula''. We will see that this simple viewpoint allows us to get the notion of a possibly infinite dimensional H-type group.

We fix some notions that we will use throughout the paper.
Let $\Ht$ be a Hilbert space, consider a continuous Lie product
$[\cdot,\cdot]:\Ht\times \Ht\to\Ht$ 
and two orthogonal and nontrivial closed subspaces
$\Ha$ and $\Hb$ such that 
$\Ht=\Ha\oplus\Hb$, with $\dim(\Hb)<+\infty$.
We denote the scalar product on $\Ht$ by $\lan\cdot,\cdot\ran$ and the 
associated norm by $|\cdot|$. 
The space of all linear continuous
endomorphisms of a Banach space $X$ 
is denoted by $\End(X)$.

We say that $\Ht$ is a {\em Hilbertian H-type group}, or simply an {\em H-type group}, if 
the following conditions hold.
\begin{enumerate}
	\item[(I)] 
	$[\Ht,\Ht]\subset \Hb$ and   $[\Ht,\Hb]=\set{0}$,
	\item [(II)]
	the unique linear and continuous operator
	$J:\Hb\to \End(\Ha)$
	defined by the formula
	\beq\label{eq:J_zxy}
	\langle J_zx,y\rangle=\ban{z,[x,y]}
	\eeq
	for $z\in \Hb$, $x,y\in\Ha$, satisfies the additional condition 
	\begin{equation}\label{eq:J_z2}
		J_z^2=-|z|^2\Id_\Ha,
	\end{equation}
where $\Id_\Ha:\Ha\to\Ha$ is the identity mapping.
\end{enumerate}
Notice that the existence of the linear and continuous operator $J$ is a consequence of both the Riesz representation theorem applied on $\Ha$ and 
the continuity of the bilinear mapping $[\cdot,\cdot]$. Thus, \eqref{eq:J_zxy}
immediately follows.

The group operation is automatically obtained
by the BCH formula:
\begin{equation}\label{eq:GroupOp}
	x\cdot y
	=x+y+\frac{1}{2}[x,y].
\end{equation}
From the defining formula \eqref{eq:J_zxy}, we immediately notice that the adjoint operator $J_z^*$ satisfies
\beq
J_z^*=-J_z.
\eeq
As a consequence, using also \eqref{eq:J_z2}, we may write
\[
|J_zx|^2=-\lan x,J_z^2x\ran=|z|^2|x|^2,
\]
that gives 
\beq\label{eq:J_znorm}
|J_zx|=|z||x|.
\eeq
Therefore, using also the defining formula \eqref{eq:J_zxy}, we have
\beq\label{eq:z^2x^2}
|z|^2|x|^2=|J_zx|^2=\lan z,[x,J_zx]\ran.
\eeq
For every $w\in \Hb$, we also notice that \eqref{eq:J_znorm} implies
\[
|\lan w,[x,y]\ran|=|\lan J_wx,y\ran|\le |w|\,|x|\,|y|,
\]
therefore in any H-type group we have 
\beq\label{eq:[]1}
|[x,y]|\le |x|\,|y|.
\eeq
For $x,z\neq0$, it follows that
\beq
\Big|\frac{[x,J_zx]}{|x|^2|z|}\Big|\le 1
\eeq
and in addition \eqref{eq:z^2x^2} gives
\beq\label{eq:par1}
\ban{\frac{z}{|z|},\frac{[x,J_zx]}{|x|^2|z|}}=1.
\eeq
Combining \eqref{eq:[]1} and \eqref{eq:par1},
we have proved that 
\beq\label{eq:xJ_z}
[x,J_zx]=|x|^2z.
\eeq
Notice that $[\Ht,\Hb]=\set{0}$ implies that $\Hb$ is contained in the center of $\Ht$, where we regard $\Ht$
as a Lie algebra. 
However, it is easy to notice that condition \eqref{eq:xJ_z} shows that 
$\Hb$ exactly coincides with the center of $\Ht$.

\begin{remark}\em 
Notice that in the case $\dim(\Ha)<+\infty$, the Hilbertian H-type
group coincides with the well known (finite dimensional) H-type group, \cite{Kaplan1980}, hence motivating our terminology.
\end{remark}

Next, we construct examples of (infinite dimensional) Hilbertian H-type groups. We fix an H-type group $\cn=\upsilon\oplus\zeta$,
where $\upsilon$ and $\zeta$ are finite dimensional
orthogonal subpaces of the Hilbert space $\cn$.
We denote by $\ban{\cdot,\cdot}_\cn$ the scalar product of $\cn$ and by $|\cdot|_\cn$ its associated norm.

The endomorphism $J^\cn:\zeta\to\End(\upsilon)$
defines the H-type structure on $\cn$. 
We denote by $\N^+$ the set of positive integers and consider the space of square-summable sequences
\beq
\Pds_\upsilon=\set{(x_k)_k:\,x_k\in \upsilon, k\in\N^+,\sum_{k=1}^\infty |x_k|_\cn^2<+\infty}.
\eeq 

We set $\Ht=\Pds_\upsilon\times \zeta$ and identifying $\Pds_\upsilon$ and $\zeta$ with $\Pds_\upsilon\times\set{0}$ and $\set{0}\times\zeta$, respectively, we can also write 
\[
\Ht=\Pds_\upsilon\oplus\zeta.
\]
For $(x,z),(x',z')\in\Ht$, we define the scalar product
\beq\label{eq:scalar}
\lan(x,z),(x',z')\ran= \lan ((x_k)_k,z),((x'_k)_k,z')\ran=
\ban{z,z'}_\cn+\sum_{k=1}^\infty \lan x_k,x'_k\ran_\cn
\eeq
that makes $\Ht$ a Hilbert space,
where $\Pds_\upsilon$ and $\zeta$ 
are orthogonal closed subspaces.
We denote by $|\cdot|$ the associated norm on $\Ht$.
For $x=(x_k)_k\in\Pds_\upsilon$ and $z\in\zeta$, we define
\beq\label{eq:defJ_z}
J_z(x)=(J^\cn_zx_k)_k.
\eeq
Thus, observing that 
\beq\label{eq:J_z^2norm}
\sum_{k=1}^\infty |J^\cn_zx_k|_\cn^2=|z|_\cn^2
\sum_{k=1}^\infty |x_k|_\cn^2<+\infty,
\eeq
the mapping $J_z:\Pds_\upsilon\to\Pds_\upsilon$ is well defined and 
\beq 
J_z^2=-|z|_\cn^2\Id_{\Pds_\upsilon},
\eeq
since $(J^\cn_z)^2=-|z|_\cn^2\Id_\upsilon $.
The Lie product of $\xi+\eta,\xi'+\eta'\in\cn=\upsilon\oplus\zeta $ 
is given by a skew-symmetric continuous bilinear mapping
\[
\beta:\upsilon\times \upsilon\to \zeta
\]
such that
\[
[\xi+\eta,\xi'+\eta']=\beta(\xi,\xi').
\]
By the property \eqref{eq:[]1} for H-type groups, we get
\beq
|\beta(\xi,\xi')|=|[\xi,\xi']|\le |\xi|_\cn\,|\xi'|_\cn
\eeq
for all $\xi,\xi'\in\upsilon$, therefore
the Lie product 
\beq\label{eq:LieHtype}
[(x,z),(x',z')] 
= 
\left(0,\sum_{k=1}^{+\infty} \beta(x_k,x'_k)
\right),
\eeq
is well defined for all $(x,z),(x',z')\in\Ht$.
Cauchy--Schwarz inequality yields
\beq
|[(x,z),(x',z')]|\le 
\sum_{k=1}^{+\infty} |\beta(x_k,x'_k)|\le 
\sum_{k=1}^\infty |x_k|_\cn\,|x_k'|_\cn\le |x|\,|x'|,
\eeq
hence the Lie product $[\cdot,\cdot]$ is continuous
on $\Ht$. 
Finally, from definition \eqref{eq:defJ_z} of $J_z:\Pds_\upsilon\to\Pds_\upsilon$, we obtain
\[
\ban{J_zx,y}=\sum_{k=1}^\infty \lan J^\cn_zx_k,y_k\ran_\cn
=\sum_{k=1}^\infty \lan z,[x_k,y_k]\ran_\cn
=\sum_{k=1}^\infty \lan z,\beta(x_k,y_k)\ran_\cn
=\ban{z,[x,y]}
\]
for all $x,y\in\Pds_\upsilon$ and $z\in\zeta$. 
We have proved the following result.

\begin{theorem}\label{th:ExHtype}
	The linear space $\Ht=\Pds_\upsilon\oplus\zeta$ 
	equipped with scalar product \eqref{eq:scalar},
	Lie product \eqref{eq:LieHtype} and
	linear operator \eqref{eq:defJ_z} is an infinite dimensional H-type group.
\end{theorem}

\begin{remark}\label{rem:infiniteHtype}\em 
	By \cite[Corollary~1]{Kaplan1980}, there exist infinitely many finite dimensional H-type groups, where there are no isomorphic couples. Indeed, these groups can be chosen to have centers of different dimensions. As a result, Theorem~\ref{th:ExHtype} 
	also shows that there are infinitely many infinite dimensional 
	H-type groups.
\end{remark}

\begin{remark}\em 
	We point out that, when the finite dimensional H-type group $\cn$ coincides with the 3-dimensional Heisenberg group, in the construction of $\Ht$, then Theorem~\ref{th:ExHtype} yields the infinite dimensional
	Heisenberg group studied in \cite{MagTib2023}.
\end{remark}


\section{Weak metrics on Hilbertian H-type groups}\label{section:weakmetrics}

In the sequel, $\Ht=\Ha\oplus\Hb$ always denotes a Hilbertian H-type group,
equipped with a scalar product $\langle\cdot,\cdot\rangle$ and its Hilbertian norm $|\cdot|$. 
This section presents various notions of weak metrics on $\Ht$. 
They include weak Finsler metrics and weak Riemannian metrics.
Indeed, both of these metrics may induce a topology which is strictly weaker than the manifold topology.
We will also follow the convention of identifying the tangent space $T_q\Ht$ with
the group itself $\Ht$, $q\in\Ht$, due to the linear structure of $\Ht$.

For every $p\in\Ht$,
the {\em left multiplication by $p$} is denoted by $L_p:\Hla\to\Hla$, with
\[
L_p(q)=p\cdot q=p+q+\frac12[p,q]
\]
for all $q\in\Ht$. We define the skew-symmetric bilinear function $\beta:\Ha\times\Ha\to \Hb$ such that
\[
[x,y]=\beta(x,y)
\]
for every $x,y\in\Ha$.
By definition of $\Ht$, we have
two canonical projections
$\pi_1:\Ht\to\Ha$ and $\pi_2:\Ht\to\Hb$ such that every $p\in\Ht$ can be written in a unique way as
\[
p=\pi_1(p)+\pi_2(p)
\]
where $\pi_1(p)$ and $\pi_2(p)$ are also orthogonal. 
We obviously have the isometric isomorphism 
\[
\Ht\to \Ha\times \Hb,\quad p\to (\pi_1(p),\pi_2(p))
\]
with respect to the Hilbert structure of $\Ht$.
We use the simplified notation $p_i=\pi_i(p)$ for $p\in\Ht$, so that
we can write $p=p_1+p_2$ with $p_1\in\Ha$ and $p_2\in\Hb$.
Then the group operation \eqref{eq:GroupOp} gives a simple formula for the 
differential of $L_p$ at a point $q\in\Ht$ along $v=v_1+v_2\in\Ht$: 
\begin{equation}\label{eq:dL_p}
	{(dL_p)}_q(v)
	=\frac{d}{dt}L_p(q+tv)\Big|_{t=0}=v +\frac{1}{2}[p,v]
	=v_1+v_2 + \frac{\beta(p_1,v_1)}{2}.
\end{equation}
Indeed, the linear structure of $\Ht$ allows us to identify $T_q\Hla$ with $\Hla$.

\subsubsection{Weak Finsler metrics and Finsler distances} \label{sect:Finsler}
We fix a norm $F_0:\Ht\to [0,+\infty)$ with respect to the linear structure of $\Ht$,
which also yields a Finsler metric on $T\Ht$. 
We always assume that $F_0$ is continuous, namely
$F_0(v)\le c_1 |v|$ for some $c_1>0$ and for all $v\in\Ht$,
where $|\cdot|$ is the fixed scalar product on $\Ht$.
It is also natural to assume that the decomposition $\Ha\oplus\Hb$ 
is compatible with the Finsler norm, namely $\pi_1:\Ht\to\Ha$ and $\pi_2:\Ht\to\Hb$ 
are continuous with respect to $F_0$. 
In other words, there exists $C>0$ such that
\[
F_0(\pi_1(x))\le C F_0(x)\quad\text{and} \quad F_0(\pi_2(x))\le C F_0(x).
\]
Thus, for each $p\in\Ht$, we set
\[
F_p(v)=F_0((dL_{-p})_p(v))
\] 
for every $v\in T_p\Ht$. 
We say that the map $F$ on $T\Ht$ arising from 
the norms $F_p$ is a {\em weak, left invariant Finsler metric} on 
$T\Ht$.
We say that $F$ is a {\em strong, left invariant Finsler metric} if
the topology induced by $F_0$ on $\Ht$ coincides with the already given
Hilbert topology of $\Ht$. 
In different terms, there exist $\tilde c_1>0$ such that
$F_0(v)\ge \tilde c_1|v|$ for all $v\in\Ht$. 
If a weak, left invariant Finsler metric $F$ on $\Ht$ is not strong,
then we say that $F$ is a {\em strictly weak, left invariant Finsler metric} on $\Ht$.

 \begin{example}\em 
Let us consider the infinite dimensional Heisenberg group $\H=\ell^2\times\ell^2\times\R$ 
	equipped with the product of the associated Hilbert structure and the group operation as defined in \cite{MagTib2023}. We have $\H=\Ha\oplus\Hb$, where $\Ha=\ell^2\times\ell^2\times\set{0}$ and $\Hb=\set{0}\times\set{0}\times\R$. 
	We fix $\cp>2$ and for an element $(h,k,t)\in\Ht$, 
	we define the norm
	\[
	F_0(h,k,t)=\|h\|_\cp+\|k\|_\cp+|t|,
	\]
	where $\|x\|_\cp=(\sum_{k=1}^\infty |x_j|^\cp)^{1/\cp}\le \|x\|_2<+\infty $ for every $x\in\ell^2$. Clearly $F_0$ gives an example of 
	strictly weak, left invariant Finsler metric.
	Indeed, it is also obvious that the projections $\pi_1$ and $\pi_2$ 
	on $\Ha$ and $\Hb$ are $F_0$-continuous, respectively. 
\end{example}

The length of a continuous, piecewise smooth curve $\gamma:[0,1] \to \Ht$ 
is defined by the integral
\begin{equation}\label{def:l_F}
\ell_F(\gamma) 
=  \int_{0}^{1} 
F_{\gamma(t)} (\dot{\gamma} (t)) dt=
\int_0^1
F_0 \left (
{(dL_{-\gamma(t)})
}_{\gamma(t)}\dot{\gamma}(t) 
\right )\,dt.
\end{equation}

Then we can immediately define the associated {\em Finsler distance} 
\begin{equation}\label{defrho}
	d_F(p,q)=\inf\{ \ell_F(\gamma): \gamma \text{ is continuous, piecewise smooth, } \gamma(0)=p \text{ and } \gamma(1)=q\}
\end{equation}
for every $p,q\in\Ht$, hence $d_F: \Ht\times\Ht\to [0,+ \infty)$.
Clearly $d_F$ is left invariant, symmetric and satisfies the triangle inequality.

\begin{remark}\label{finsler:positive}\em 
	Let us consider a weak, left invariant Finsler metric $F$ on $\Ht$, 
	and let $d_F$ be the associated geodesic distance. 
	We will prove that for $p,q\in \Hla$ with $\pi_1(p)=x\neq y=\pi_1(q)$,
	we have $C\,d_F(p,q)\ge F_0(x-y) >0$.
	Indeed, for every continuous, piecewise smooth curve $\gamma :[0,1] \to \Ht$ joining $p$ to $q$, we get	
	\[
	\ell_F(\gamma)=\int_0^1
	F_0\Big((dL_{-\gamma(t)})_{\gamma(t)}
	(\dot{\gamma}(t))\Big)
	\,dt\ge\frac1C \int_0^1F_0(\dot{\gamma_1}(t))
	\,dt
	\]
	in view of \eqref{eq:dLp} and taking into account the $F$-continuity of the projections. Thus, if we consider the projected curve $\gamma_1:[0,1]\to\Ha$,
	we can piecewise integrate $\dot\gamma_1$ on the intervals where it is continuous. Then we apply \cite[Theorem 2.1.1 (ii)]{Hamilton1982} and \cite[Theorem 2.2.2]{Hamilton1982} and the triangle inequality on a partition $t_0=0<t_1<\cdots<t_k=1$.
	It follows that
	\begin{align*}
	\ell_{ F}(\gamma)&\ge\frac{1}{C}\int_0^1F_0(\dot{\gamma_1}(t))\,dt\ge
	\frac1C\sum_{j=0}^{k-1}F_0\left(\int_{t_j}^{t_{j+1}}\dot\gamma_1(t)\right)=\frac1C\sum_{j=0}^{k-1}F_0\left(\gamma_1(t_{j+1})-\gamma_1(t_j)\right) \\
&\geq\frac{F_0(x-y)}{C}>0.
	\end{align*}
\end{remark}

\subsubsection{Weak sub-Finsler metrics}
Identifying $\Ha \oplus \Hb$ with $T_0(\Ha\oplus\Hb)$, 
the subspace $\Ha$ can be seen as a closed subspace of
$T_0\Hla$, that we denote by $H_0\Hla$ and we may introduce the \textit{left invariant horizontal subbundle}, denoted by $H\Hla$, with fibers 
$$
H_p\Hla = (dL_p)_0 (H_0\Hla)\subset T_p\Hla
$$
for every $p\in\Hla$.
For each $p\in\Ht$, on the horizontal fiber $H_p\Hla$ of $H\Hla$ we can fix a norm, which turns out to be continuous and left invariant.
Precisely, a \textit{weak, left invariant sub-Finsler metric} $S$ on $H\Hla$ is defined
by a norm 
\beq\label{eq:g_0subR}
S_0: \Ha \to [0,+\infty)
\eeq
satisfying for some $c_0>0$ and for all $x \in \Ha$ the inequality 
\begin{equation}
S_0(x) \leq c_0 |x|.
\end{equation}
The previous condition immediately yields the continuity of $S_0$ with respect to the
fixed Hilbert topology on $\Ht$. Notice that the closed subspace $\Ha$ inherits a Hilbert 
structure from $\Ht$.
With the previous identifications, for every $p \in\Ht$ and $v \in H_p \Ht$, we introduce the norm
\begin{equation} \label{leftinvmetric:def}
	S_p(v)	=
	 S_0 \left ( ({dL_{-p}})_p(v) \right)
\end{equation} 
on the fiber $H_p\Hla$.
If the topology defined by the norm $S_0$ on $\Ha$ coincides with the Hilbert one of $\Ha$, 
we say that $S_0$ defines a {\em strong, left invariant sub-Finsler metric}. This is equivalent to the existence of a constant $\tilde{c} >0$ such that
 $\tilde{c} |x| \leq S_0 (x)$ for all $ x \in \Ha$.
 If this is not the case, we say that $S_0$ defines a {\em strictly weak, left invariant sub-Finsler metric}.
 
 \begin{example}\em 
 	Let us consider the infinite dimensional Heisenberg group $\H=\ell^2\times\ell^2\times\R$ 
 	equipped with the product of the associated Hilbert structure and the group operation
 	as defined in \cite{MagTib2023}. We have $\H=\Ha\oplus\Hb$, where $\Ha=\ell^2\times\ell^2\times\set{0}$ and $\Hb=\set{0}\times\set{0}\times\R$. 
 	We fix $\cp>2$ and for an element $(h,k,0)\in\Ha$, 
 	we define the norm
 	\[
 	S_0(h,k)=\|h\|_\cp+\|k\|_\cp,
 	\]
 	where $\|x\|_\cp=(\sum_{k=1}^\infty |x_j|^\cp)^{1/\cp}\le \|x\|_2<+\infty $ for every $x\in\ell^2$. Clearly $S_0$ gives an example of 
 	strictly weak, left invariant sub-Finsler metric.
\end{example}
 \subsubsection{Horizontal curves and sub-Finsler distances} 
We notice that the expression of the differential of translations
\eqref{eq:dL_p} proves that $v \in H_p\Hla$ if and only if 
 \begin{equation}\label{eq:dLp}
 	(dL_{-p})_p(v)  =v - \frac{1}{2}[p,v] =v_1+ v_2 - \frac{\beta(p_1,v_1) }{2} \in H_0\Hla
 \end{equation}
 and the previous condition corresponds to the equality 
\begin{equation}\label{eq:betahoriz}
 v_2 = \frac{\beta(p_1,v_1)}{2}.
\end{equation}
Thus, we have a precise formula to define the {\em horizontal curves}
 associated with $H\vH$. They are continuous and piecewise smooth curves 
 $\gamma:[0,1] \to\Hla$
 of the form $\gamma=\gamma_1+\gamma_2\in\Hla$, such that 
 for almost every $t \in [0,1]$ we have 
 $$
 \dot{\gamma}_2(t)=\frac{\beta(\gamma_1(t), \dot{\gamma}_1(t))}{2}.
 $$
 The previous differential constraint means that $\dot{\gamma}(t) \in H_{\gamma(t)}\Hla$.
 The length of a horizontal curve $\gamma:[0,1] \to \Ht$ is defined by $\ell_S(\gamma) 
 =  \int_{0}^{1} 
 S_{\gamma(t)} (\dot{\gamma} (t)) dt$, therefore
$$\ell_S(\gamma)=\int_0^1
S_0 \left (
{(dL_{-\gamma(t)})
}_{\gamma(t)}\dot{\gamma}(t) 
\right )\,dt	
=
\int_0^1
	S_0 
		( 
	\dot{\gamma}_1(t)
		  )
		 \,dt.$$		
It is not difficult to observe that all couple of points in $\Hla$ can be connected by horizontal curves. As a result, the \textit{sub-Finsler distance} 
\begin{equation*}
	\rho_S(p,q)=\inf\{ \ell_S(\gamma): \gamma \text{ is a horizontal curve with } \gamma(0)=p, \gamma(1)=q\}
\end{equation*}
is finite for every $p,q\in\Ht$,
hence $\rho_F: \Ht\times\Ht\to [0,+ \infty)$.
The fact that $\rho_F$ is left invariant, symmetric and satisfies the triangle inequality is straightforward.

\begin{remark}\label{subfinsler:positive}\em 
Let us consider a weak, left invariant sub-Finsler metric $S$ on $\Ht$, 
and a weak, left invariant Finsler metric $F$ on $\Ht$ such that $F_0|_{\Ha}=S_0$. 
We define $\rho_S$ and $d_F$ to be the associated sub-Finsler distance
and Finsler distance, respectively. 
Taking into account \eqref{def:l_F}, \eqref{eq:dLp} and \eqref{eq:betahoriz} we observe that 
$\ell_F(\gamma)=\ell_S(\gamma)$ for every horizontal curve. 
Then we immediately get 
$$
\rho_S(p,q)\ge d_F(p,q)
$$
for every $p,q\in \Hla$.
Taking into account Remark~\ref{finsler:positive} 
we also have $\rho_S(p,q)\ge d_F(p,q)>0$ whenever $\pi_1(p)\neq\pi_1(q)$. 
Notice that for any fixed weak sub-Finsler metric $S_0$ on $\Ht$, 
we can always find a weak Finsler metric $F_0$ such that $F_0|_\Ha=S_0$.
It suffices to choose any Hilbert norm $|\cdot|$ on $\Hb$, defining 
\[
F_0(x+z)=S_0(x)+|z|
\]
for every $x\in\Ha$ and $z\in\Hb$. 
\end{remark}

\subsubsection{Weak Riemannian metrics and Riemannian distances}\label{ssub:weakR}
Following Section~\ref{sect:Htype}, we consider a Hilbertian H-type group $\Hla= \Ha \oplus \Hb$ equipped with a Hilbert product $\ban{\cdot,\cdot}$ and the mapping $J_z$, $z\in\Hb$.
We fix a continuous scalar product $\sigma_0$ on $\Ht$, namely
\begin{equation}\label{ContinuityRiemannianMetric}
	\|v\|_{\sigma_0} \leq c_0 |v|
\end{equation}
for some $c_0>0$ and every $v\in\Ht$, where $\|\cdot\|_{\sigma_0}$ is 
the norm arising from $\sigma_0$.
We also require that the canonical projections
$\pi_1:\Ht\to\Ha$ and $\pi_2:\Ht\to\Hb$ are $\sigma_0$-continuous, that is
\[
\|\pi_1(v)\|_{\sigma_0} \le C \|v\|_{\sigma_0}
\quad \text{and}\quad \|\pi_2(v)\|_{\sigma_0} \le C \|v\|_{\sigma_0}
\]
for all $v\in\Ht$ and some $C>0$. 
Thus, $\sigma_0$ gives a scalar product
\begin{equation} \label{leftinvmetric:defr}
	\sigma_p(v,w) = \sigma_0 \big((dL_{p^{-1}})_pv , (dL_{p^{-1}})_pw \big)
	=\sigma_0 \big(({dL_{-p}})_pv , ({dL_{-p}})_pw \big)
\end{equation} 
for each $p \in\Ht$ and $v,w\in T_p\Ht$.
The corresponding Riemannian metric $\sigma$ on $T\Ht$ is called
{\em weak, left invariant Riemannian metric}.
Notice that the Riemannian norm $\|\cdot\|_{\sigma_0}$ on $\Ht$ is also Finsler metric, see Section~\ref{sect:Finsler}. 

Let us consider the topology defined by $\sigma_0$ on $\Ht$. When it coincides with the topology determined by the Hilbert structure of $\Ht$, we say that $\sigma$ is a 
{\em strong, left invariant Riemannian metric}.
We say that $\sigma$ is a {\em strictly weak, left invariant Riemannian metric} if it is not strong.
Finally, a (strictly) weak, left invariant Riemannian metric $\sigma$ on $\Ht$ such that
$\Ha$ and $\Hb$ are $\sigma_0$-orthogonal is called {\em (strictly) weak, graded Riemannian metric}.

For a fixed weak, left invariant Riemannian metric $\sigma$ on $\Hla$, we consider the linear and continuous operator $A:\Hla\to\Hla$ such that for all $v,w \in \Hla$ we have
$\sigma_0(v,w)=\langle v,Aw \rangle$, where $\langle \cdot, \cdot \rangle$ denotes 
the Hilbert product on $\Ht$. 
The operator $A$ exists by the classical Riesz representation theorem
and it is automatically self-adjoint and positive.

We denote by $A_{\Ha} $ its restriction to $\Ha$ and by $A_{\Hb}$ its restriction to $\Hb$. 
When $\sigma_0$ is graded, it is easy to notice that $A_\Ha(\Ha)\subset \Ha$
and $A_\Hb(\Hb)\subset\Hb$.  
Then we can consider the linear and continuous operators 
$A_{\Ha} : \Ha \to \Ha$ and 
$A_{\Hb} : \Hb \to \Hb$.

The following proposition is also standard.

\begin{proposition}\label{ImmagineOperatoreAssociatoMetrica}
	If $\sigma$ is a weak, left invariant Riemannian metric on $\Ht$, 
	then the subspace $A(\Ht)$ is dense in $\Ht$. 
	Furthermore, $\sigma$ is strong if and only if $A$ is surjective.
\end{proposition}

For any continuous and piecewise smooth curve
$\gamma:[0,1]\to \Ht$ its Riemannian length with respect to the
weak Riemannian metric $\sigma$ is defined as
\[
\ell_\sigma(\gamma) = \int_{0}^{1} \norm{ \dot{\gamma} (t) }_\sigma dt.
\]
The \textit{geodesic distance} associated with $\sigma$ 
is the function $d_\sigma: \Hla\times\Hla\to [0,+ \infty)$ defined as
\[
d_\sigma(p,q)=\inf\{ \ell_\sigma(\gamma): \gamma \text{ is a continuous and piecewise smooth curve with } \gamma(0)=p, \gamma(1)=q\}.
\]
Clearly $d_\sigma$ is left invariant, symmetric and it satisfies the triangle inequality.

\section{Degenerate geodesic distances}\label{section:Riemannian}

The next theorem proves the existence of degenerate sub-Finsler distaces 
in any Hilbertian H-type group equipped with a strictly weak, left invariant sub-Finsler metric.

\begin{theorem}[Vanishinig of sub-Finsler distances]\label{thm:vanishingsubF}
Let $\Ht=\Ha\oplus \Hb$ be an infinite dimensional H-type group 
equipped with the canonical projections $\pi_1:\Ht\to\Ha$ and 
$\pi_2:\Ht\to\Hb$.
Let $\rho_S$ be the sub-Finsler distance arising from any
strictly weak, left invariant sub-Finsler metric $S$ on $\Ht$.
Then for every $p,q\in\Ht$ with $\pi_1(p)=\pi_1(q)$, we have $\rho_S(p,q)=0$.
\end{theorem}
\begin{proof}
It suffices to prove that for all $c\in\R$ and all $z\in\Hb$ with $\vert z \vert=1$, we have 
\begin{equation}\label{eq:rho00}
\rho_S\bigg(0,\frac{c^2}{3}z\bigg)=0.
\end{equation}
Since the norm $S_0$ of \eqref{eq:g_0subR} does not define the Hilbert topology of $\Ha$, there exists a sequence ${\{w^n\}}_n$ in $\Ha$ such that 
$\vert w^n\vert=1$ and 
$S_0(w^n) \leq \frac{1}{n}$ for all $n\in\N^+$.
We choose $z\in\Hb$ with $\vert z\vert=1$, and for each $n \in \N$, define
$$
\gamma_1^n(t)= tc\sqrt{n}w_n+\frac{t^2c}{2}\frac{1}{\sqrt{n}}J_z(w_n).
$$
Consider now the curve $\gamma^n = (\gamma^n_1,\gamma^n_2)$, where  
$$
\gamma^n_2(t)= 
\frac{1}{2}
\int_0^t
\beta(\gamma^n_1(s),\dot{\gamma}^n_1(s)) \, ds\in\Hb.
$$
By construction, the curve $\gamma^n$ is horizontal, therefore
$\ell_S(\gamma^n)
=
\int_0^1S_0(\dot{\gamma}_1^n(t))\,dt$.
Let us consider the following estimates
\begin{equation}
	\begin{split}
		\ell_S(\gamma^n)=&
		\int_0^1 S_0(
		{\dot{\gamma}}_1^n(t)
		) \,dt
		=
		\int_0^1
			S_0\left (c\sqrt{n}w_n+\frac{ct}{\sqrt{n}}J_z(w_n) \right)
			 \,dt\\
		\leq &
		c\sqrt{n}\,
		S_0(w_n) 
		+
		\frac{c}{\sqrt{n}}
		S_0(J_z(w_n))\leq 
		\frac{c}{\sqrt{n}}
		+
		\frac{cc_0}{\sqrt{n}}
		|J_z(w_n)|			
		=
		\frac{c}{\sqrt{n}}
		+
		\frac{c c_0}{\sqrt{n}}\cdot 
		|z|.
	\end{split} 
\end{equation}
It follows that 
\[ \lim_{n\to\infty}\ell_S(\gamma^n)=0. \] 
For each $n$, the endpoint of $\gamma^n$ is 
$$
\gamma^n(1)= 
c\sqrt{n}w_n
+
\frac{c}{2}
\frac{1}{\sqrt{n}}
J_z(w_n)+\frac{c^2}{12}\, z.
$$
Now, we define the curve $
\alpha^n_1:[0,1]\to \Ha$ as
$$
\alpha^n_1(t)
=
c\sqrt{n}(1-t)w_n
+
\frac{c}{2}
\frac{(1-t^2)}{\sqrt{n}}
J_z(w_n)
$$
and consider the lifting $\alpha^n=\alpha^n_1+\alpha^n_2$, where 
$$
\alpha^n_2(t)=\gamma^n_2(1)+
\frac{1}{2}\int_0^t\beta(\alpha^n_1(s),\dot{\alpha}^n_1(s))\,ds\in\Hb
$$
By construction, $\alpha^n$ is also horizontal and $\alpha^n(0)=\gamma^n(1)$, therefore the curve $ \alpha^n\star\gamma^n $ obtained by joining $\gamma_n$ and $\alpha_n$ is also horizontal.
For each $n\in \N$, the curve $ \alpha^n\star\gamma^n $ connects  the origin $0\in\Ht$ to the point $\frac{c^2z}{6}\in\Hb$.
We finally observe that
$$
\ell_S(\alpha^n)=
\int_0^1
S_0 (\dot{\alpha}_1^n(t))
\,dt
=
\int_0^1
S_0 \left(
		c\sqrt{n}w_n
		+
		\frac{ct}{\sqrt{n}}J_z(w_n) \right)
\,dt
=
\ell_S(\gamma^n) \to 0.
$$
Therefore, $\ell_S(\alpha_n\star\gamma_n)= \ell_S(\gamma_n) + \ell_S(\alpha_n) \to 0$. 
We have proved that \eqref{eq:rho00} holds
for every $c\in\R$ and $z\in\Hb$. 
To conclude the proof, we consider $z_1,z_2\in\Hb$, $z_1\neq z_2$ and $x\in\Ha$. 
We notice that the left invariance of the sub-Finsler distance function
yields
\[
\rho_S(x+z_1,x+z_2)=\rho_S(xz_1,xz_2)=\rho_S(z_1,z_2)=\rho_S(0,z_2-z_1).
\]
Clearly, we can find $c\neq0$ and $z\in\Hb\setminus\set{0}$ such that 
$z_2-z_1=c^2z/6$, hence 
\[
\rho_S(x+z_1,x+z_2)=\rho_S(0,c^2z/6)=0,
\]
concluding the proof.
\end{proof}

\begin{corollary}\label{cor:vanishingsubFinslerdistances}
	Let us fix a strictly weak, left invariant sub-Finsler metric $S$ on a 
	Hilbertian H-type group $\Ht = \Ha \oplus \Hb$. 
	Then for $x,y \in \Ha$ and $z_1 , z_2 \in \Hb$, we have 
	\[
	\rho_S(x+z_1,y+z_2)=0\quad \text{if and only if}\quad x=y,
	\]
	where $\rho_S$ is the sub-Finsler distance associated with $S$. 
\end{corollary}
The main implication of this corollary follows by Theorem~\ref{thm:vanishingsubF}.
The full characterization of the two conditions is obtained
by showing that points with different projections on $\Ha$ must
have positive Finsler distances. 
This is a consequence of combining Remark~\ref{finsler:positive} and Remark~\ref{subfinsler:positive}.

\begin{lemma}\label{esistenzavettoridegeneri}
If $F$ be a strictly weak, left invariant Finsler metric $F$ on a
Hilbertian H-type group $\Ht = \Ha \oplus \Hb$. Then there exists a sequence
$\{h_n\}_{n \in \N} \subset \Ha$ such that $F_0(h_n)\to 0$ and $|h_n| =1$ for all $n\in\N$. 
\end{lemma}
\begin{proof}
The topology defined by $F_0$ on $\Ht$ is not the Hilbert one, therefore there exists a sequence $u_n$ in $\Ht$ such that $|u_n|=1$ for all $n$ and $F_0(u_n)\to0$. We can write 
\[
u_n= v_n+w_n=\pi_1(u_n)+\pi_2(u_n),
\]
where $v_n \in \Ha$ and $w_n \in \Hb$.
By the continuity of the projections,
$CF_0(u_n)\ge F_0(w_n)$ therefore $F_0(w_n)\to0$.
Since $\Hb$ is finite dimensional, we also have
$|w_n|\to0$, therefore $|v_n|\to1$.
Again the continuity of the projections yields 
$F_0(v_n)\to0$. 
To conclude the proof, we consider a subsequence $v_n$ of nonzero vectors, and we observe that the renormalized sequence $h_n=\frac{v_n}{ |v_n| }$ satisfies our claim.
\end{proof}

\begin{theorem}\label{thm:vanishingFinslerdistances}
Let $F$ be a strictly weak, left invariant Finsler metric on a Hilbertian H-type group $\Ht = \Ha \oplus \Hb$. 
Then for every $x,y \in \Ha$ and $z_1 , z_2 \in \Hb$, we have 
$d_F((x,z_1),(y,z_2))=0$ if and only if $x =y$.
\end{theorem}
\begin{proof}
The restriction of $F_0$ to $\Ha$ defines a weak sub-Finsler metric
$S_0:\Ha\to[0,+\infty)$. By Lemma~\ref{esistenzavettoridegeneri}, the 
corresponding left invariant sub-Finsler metric $S$ is strictly weak. 
In view of Remark~\ref{subfinsler:positive}, 
we have $\rho_S\ge d_F$, so we can apply Theorem~\ref{thm:vanishingsubF},
obtaining that $d_F(p,q)=0$, whenever $\pi_1(p)=\pi_1(q)$. By Remark~\ref{finsler:positive}, the proof is complete.
\end{proof}

\begin{corollary}\label{cor:vanishingRiemanniandistances}
Let $\sigma$ be a strictly weak, left invariant Riemannian metric on a Hilbertian H-type group $\Ht = \Ha \oplus \Hb$. 
	Then for every $x,y \in \Ha$ and $z_1 , z_2 \in \Hb$, we have 
	$d_\sigma((x,z_1),(y,z_2))=0$ if and only if $x =y$.
\end{corollary}

The previous corollary follows by observing that a sitrictly weak, left invariant Riemannian metric also yields a strictly weak, left invariant Finsler metric.

\section{Non-existence of the Levi-Civita covariant derivative}\label{sec:nonexLC}

In this section, we fix a Hilbertian H-type group $\Ht$ with its Lie product $[\cdot , \cdot]$. We consider the Lie algebra $\Lie \Ht$ of left invariant vector fields on $\Ht$. The associated Lie product is the skew-symmetric bilinear mapping $[\cdot , \cdot]: \Lie \Ht \times \Lie \Ht \to \Lie \Ht$ its Lie product. In our setting, the linear structure of $\Ht$ allows us also consider 
the "identification" $\kI: \Ht \to \Lie \Ht$,
where $\kI(v)=X_v$ is the unique left invariant vector field of $\Lie(\Ht)$ such that $X_v(0)=v$. 
In fact, there is the already mentioned identification
between $T_0\Ht$ and $\Ht$.
Throughout the section, the continuous linear 
and self-adjoint operator $A:\Ht\to\Ht$ is defined
by the weak metric $\sigma_0(v,w)=\ban{v,Aw}$
for $v,w\in\Ht$.

The first result of this section is to prove that the Lie algebra $\Lie(\Ht)$ is actually isomorphic to the starting Lie algebra $\Ht$, and the isomorphism is given by the map $\kI$. 

\begin{proposition}\label{IsoAlgebreLie}
Let $\Ht$ be an H-type group. Then the map $\kI$ is a Lie algebra isomorphism, that is, for every $x, y \in \Ht$ we have $\kI_{[x , y]}=[\kI_x ,\kI_y]$. 
\end{proposition}

The proof of the previous proposition can be obtained by standard arguments, taking into account that the group operation in $\Ht$ is given by the BCH formula and the Lie product on $\Ht$. Actually, it holds in general Banach nilpotent Lie groups, \cite[Proposition~2.1]{MagRaj2014}.

\begin{theorem}\label{CondizioneNecessariaEsistenzaLeviCivita}
Let $\sigma$ be a weak, graded Riemannian metric on a Hilbertian H-type group $\Ht = \Ha \oplus \Hb$. If $\sigma$ admits the Levi-Civita covariant derivative $\nabla$, then for every $x= x_1 + x_2 \in \Ht$ with $x_1 \in \Ha$ and $x_2 \in \Hb$ we have
\begin{equation}
 J_{A x_2} x_1 \in \im A\quad 
 \text{ and } \quad
 \nabla_{\kI_x} \kI_x (0) 
 = 
 - A^{-1} 
 \left (   
 J_{A x_2} x_1
 \right).
\end{equation} 
\end{theorem}
\begin{proof}
Suppose that $\nabla$ is the Levi-Civita covariant derivative.
Since $\nabla$ is torsion-free, we have 
$[\kI_x , \kI_y]= \nabla_{\kI_x} \kI_y - \nabla_{\kI_y} \kI_x$
for $x,y\in\Ht$. 
By the left invariance of $\kI_x$ and $\kI_y$, 
the function $\Ht\ni p \to \sigma_p(\kI_x (p),\kI_y (p)) $ is constantly equal to $\sigma_0 ( x , y )$, by the identification
of $\Ht$ with $T_0\Ht$. 
The key property of the Levi-Civita covariant derivative
yields 
\begin{equation}\label{Zsigma}
	0 = Z \sigma(\kI_x,\kI_y) 
	= 
	\sigma (\nabla_Z \kI_x, \kI_y ) 
	+
	\sigma (\kI_x , \nabla_Z \kI_y )
\end{equation}
for every $Z$ vector field on $\Ht$.
Notice that the previous equations for $x=y$ yield 
\[
\sigma (\kI_x , \nabla_Z \kI_x )=0.
\]
As a consequence, using again \eqref{Zsigma}, we get
\begin{align*}
\sigma([\kI_x ,\kI_y ],\kI_x)
&=
\sigma( \nabla_{\kI_x} \kI_y , \kI_x )
-
\sigma( \nabla_{\kI_y} \kI_x , \kI_x )
=
\sigma( \nabla_{\kI_x} \kI_y , \kI_x )  \\
&=
-
\sigma(\kI_y ,  \nabla_{\kI_x}\kI_x )
=
-
\sigma_0( y ,  \nabla_{\kI_x}\kI_x (0) ). 
\end{align*}
By Proposition~\ref{IsoAlgebreLie}, it follows that
$$ 
\sigma([\kI_x ,\kI_y],\kI_x)
=
\sigma(\kI_{[x , y ]},\kI_x)
=
\sigma_0 (\kI_{[ x , y ]} (0),\kI_x (0))
=
\sigma_0([ x , y ], x).
$$
Therefore, we have proved that
\begin{equation}
	\sigma_0 (y, \nabla_{\kI_x}\kI_x (0) )
	=
	-
	\sigma_0 ([x,y] , x),
\end{equation}
which immediately leads us to the following equalities
\begin{equation}\label{eqThmNonEsistenzaLevi}
\lan y , A \nabla_{\kI_x}\kI_x (0) \ran
=
-\lan [x,y] , Ax \ran
=
-\lan [x_1 , y_1] , A x_2 \ran
=
-\lan  y_1 , J_{A x_2} x_1 \ran.
\end{equation}
In particular,
formula \eqref{eqThmNonEsistenzaLevi} holds true for all $y \in \Hb$, hence $A \nabla_{\kI_x}\kI_x (0) \in \Ha$. 
Now, taking $y \in \Ha$ in formula \eqref{eqThmNonEsistenzaLevi} we get
$$
A \nabla_{\kI_x}\kI_x (0)
=
- J_{A x_2} x_1,
$$
which proves our claim.
\end{proof}

\begin{theorem}\label{thm:nonesistenzaLeviCivita}
Let $\sigma$ be a weak, graded Riemannian metric on an H-type group $\Ht$. If $\sigma$ is strictly weak, then it does not admit the Levi-Civita covariant derivative.
\end{theorem}
\begin{proof}
If $\sigma$ is strictly weak, then its associated operator $A$ is not surjective, by Proposition~\ref{ImmagineOperatoreAssociatoMetrica}. 
Since $\Hb$ is finite dimensional and $A_{\Hb}$ is injective, then $A_{\Hb}$ is also surjective. As a consequence, 
$A_{\Ha}$ cannot be surjective, hence
we have can choose $v \in \Ha$ such that $v  \notin A_\Ha(\Ha)$. We consider $x_2 \in \Hb $,
$x_2 \neq 0$ and we define
$x =J_{A x_2}v+x_2$, $x_1 = J_{A x_2} v$. 
By \eqref{eq:J_z2} we have
$$
J_{A x_2} x_1 = J_{A x_2}(J_{A x_2} v) = - |A x_2|^2 v \notin \im A.
$$ 
Hence, by Theorem~\ref{CondizioneNecessariaEsistenzaLeviCivita} we get
a contradiction, therefore the Levi-Civita covariant derivative does not exists for $\sigma$.
\end{proof}

Since strong Riemannian metrics always admit the Levi-Civita covariant derivative, the next corollary is straightforward.

\begin{corollary}
	Let $\sigma$ be a weak, graded Riemannian metric on an H-type group $\Ht$. Then,
	$\sigma$ admits the Levi-Civita covariant derivative if and only if it is a strong Riemannian metric.
\end{corollary}

\section{Blow-up of the sectional curvature}
We consider a Hilbertian H-type group $\Ht = \Ha \oplus \Hb$, endowed with a weak, graded Riemannian metric $\sigma$. If $\sigma$ is strong, then the sectional curvature can be computed using the Riemann tensor and the Levi-Civita covariant derivative, \cite{Klingenberg1995}. 
This approach in general does not apply when $\sigma$ is strictly weak, as a consequence of Theorem~\ref{thm:nonesistenzaLeviCivita}. 
We will also show how the Arnold's formula allows us to compute the sectional curvature for a special 
family of planes. Finally, we find a sequence of 
planes where the sectional curvatures blow-up.
 
\subsection{The \texorpdfstring{$B$}{B}-adjoint vector}

We consider the adjoint representation 
$\ad: \Ht \to \End(\Ht)$, where 
the endomorphism  $\ad_x (y)= [x , y]$ is defined by the Lie product
of $\Ht$. 
For a fixed couple of vectors $x, y \in \Ht $, we consider (in case it exists) the unique vector 
$B_{\sigma_0}(y , x) \in \Ht$ which satisfies 
the formula
\begin{equation}\label{eq:adT}
\ban{z , B_{\sigma_0}(y , x ) }_{\sigma_0}
=
\ban{ [x,z] , y }_{\sigma_0 }
\end{equation}
for every $ z \in \Ht $.
We say that $B_{\sigma_0}(y,x)$ is the 
{\em $B$-adjoint vector of $(y,x)$ with respect to $\sigma_0$}. When this vector exists, it automatically satisfies
\[
B_{\sigma_0}(ty,sx)=ts\,B_{\sigma_0}(y,x)
\]
for every $t,s\in\R$.
And also $B_{\sigma_0}(ty,sx)$ exists for some $t,s\neq0$ if and only if $B_{\sigma_0}(y,x)$ exists.
If $\sigma_0$ is a strong metric, then 
the classical Riesz representation theorem
yields the existence of $B_\sigma( y , x)$
for all $x,y\in\Ht$. 
Precisely, in this case, 
\[
B_{\sigma_0} (y, x)=\ad_x^\top(y),
\]
where $\ad_x^{\top} : \Ht \to \Ht$ is the adjoint
operator of $\ad_x $ with respect to $\sigma_0$. 
For a strictly weak Riemannian metric, the existence of $B_\sigma( y , x) \in  \Ht$ for fixed $ x, y \in  \Ht $ does not necessarily hold. 
A simple example can be found for instance in \cite{MagTib2023}.

\subsection{Arnold's formula}
To compute the sectional curvature of planes in a Hilbertian H-type group $\Ht$, 
we use the Arnold's formula
\cite[Theorem~5]{Arn66},
see also \cite{MicRat1998GeoVir}, \cite{BBM14}
and \cite{BBM2014}. 

Let us consider two {\em $\sigma_0$-orthonormal vectors $ x , y \in \Ht$}, such that the 
$B$-adjoint vectors
\[
B_{\sigma_0}( y , x), B_{\sigma_0}(x,y), 
B_{\sigma_0}( x, x), B_{\sigma_0}(y,y) \in\Ht
\]
all exist. We introduce the notation 
$\Pi_{x,y}$ to denote the vector subspace
spanned by $x$ and $y$.
The sectional curvature of $\Pi_{x,y}$
can be obtained by 
\begin{equation}\label{def:K(X,Y)Arnold}
K_\sigma(\Pi_{x,y})
=
\ban{\delta,\delta}_{\sigma_0}
+
2\ban{\alpha,\beta}_{\sigma_0}
-
3\ban{\alpha,\alpha}_{\sigma_0}
-
4\ban{B_x,B_y}_{\sigma_0}.
\end{equation}
In the previous formula we have defined 
\begin{align}
\delta
&=\frac12\pa{B_{\sigma_0}(x,y)+B_{\sigma_0}(y,x)},\quad\beta=\frac12\pa{B_{\sigma_0}(x,y)-B_{\sigma_0}(y,x)},
\quad\alpha=\frac12[x,y] \label{eq:delta_1} \\
B_x
&=\frac12B_{\sigma_0}(x,x)\quad \text{and}\quad B_y=\frac12B_{\sigma_0}(y,y). \label{eq:delta_2}
\end{align}
It is a simple computation to verify that the sectional curvature of a plane defined through this formula does not depend on the choice of the $\sigma_0$-orthonormal basis for that plane.

First of all, we provide a 
condition for which the vector $B_{\sigma} (y , x)$ exists with $x, y \in \Ht $ fixed,
see the following proposition.

\begin{proposition}[Existence of the $B$-adjoint vector]\label{CaratterizzazioneTrasposto}
Let $\sigma_0$ be a weak, graded Riemannian metric on a Hilbertian H-type group $\Ht = \Ha \oplus \Hb$
and let $ x=x_1 + x_2$, $y=y_1 + y_2 \in \Ht $, with $x_1,y_1 \in \Ha$ and $x_2 , y_2 \in \Hb$.
It follows that
\begin{equation}\label{ConditionExistenceTrasposto}
\text{there exists} \; B_{\sigma_0}(y,x) \in \Ht
\;\text{if and only if}\;
J_{A y_2} x_1 \in A_\Ha(\Ha).
\end{equation}
If one of the previous conditions holds, then 
\begin{equation}\label{FormulaPerAggiunto}
B_{\sigma_0}(y,x)= A^{-1}( J_{A y_2} x_1 ).
\end{equation}
\end{proposition}
\begin{proof}
Assume that $J_{A y_2} x_1 \in A_\Ha(\Ha)$. Thus, for all $z \in \Ht$ we have
$$
\ban{ z , A^{-1}( J_{A y_2} x_1 )  }_{\sigma_0 }
=
\langle z , J_{A y_2} x_1 \rangle
=
\langle [x , z] , A y_2 \rangle
=
\ban{[ x , z] ,y}_{\sigma_0},
$$
hence there exists 
$B_{\sigma_0}(y,x) =  A^{-1}( J_{A y_2} x_1 )$.
If $B_{\sigma_0}(y,x)\in\Ht$ exists, then for all $z \in \Ht$ we have
$$
\langle z , A (B_{\sigma_0}(y,x) )\rangle
=
\ban{ z , B_{\sigma_0}(y,x)  }_{\sigma_0}
=
\ban{ [x, z ] , y }_{\sigma_0 }
=
\langle [x_1 , z ] , A y_2 \rangle
=
\langle J_{A y_2} x_1 , z \rangle.
$$
Therefore, $A(B_{\sigma_0}(y,x) ) = J_{A y_2} x_1 $,
concluding the proof.
\end{proof}

From \eqref{ConditionExistenceTrasposto} and \eqref{FormulaPerAggiunto} we get directly $(1)$. 
From \eqref{ConditionExistenceTrasposto}, \eqref{FormulaPerAggiunto} and \eqref{eq:J_z2} we get directly $(2)$.

\begin{remark}\label{rem:Bsigma0} \em
As a consequence of Proposition~\ref{CaratterizzazioneTrasposto},
precisely of \eqref{ConditionExistenceTrasposto}, \eqref{FormulaPerAggiunto},
for all 
\[
(y,x)\in (\Ha\times \Ht)\cup (\Ht\times \Hb)
\]
we have $J_{Ay_2}x_1=0$, hence
the $B$-adjoint vector $B_{\sigma_0}(y,x)$ 
exists and it vanishes.
\end{remark}
	
\begin{remark}\label{rem:BsigmaFor} \em
For all $z \in \Hb $ and $x \in \Ha$, 
we notice that 
\[
J_{Az}\big(J_{Az}(Ax)\big)=-|Az|^2Ax\in A_\Ha(\Ha),
\]
hence \eqref{ConditionExistenceTrasposto} yields the existence of the $B$-adjoint 
vector $B_{\sigma_0}(z,J_{Az}(Ax))$
and \eqref{FormulaPerAggiunto} gives
\begin{equation}\label{eq:JAz}
B_{\sigma_0}(z,J_{Az}(Ax))
=-|Az|^2x.
\end{equation}
\end{remark}
We use Proposition~\ref{CaratterizzazioneTrasposto}
and the previous remarks to compute the sectional curvatures of specific planes,
according to the following lemma.

\begin{lemma}\label{Formulecurvaturapiani}
Let $\sigma$ be a weak graded Riemannian metric on $\Ht$. 
\begin{enumerate}
 \item If  $x , y \in \Ha $ are $\sigma_0$-orthonormal, then the sectional curvature 
 $K_{\sigma}(\Pi_{x, y})$ exists and 
 $$
 K_{\sigma}(\Pi_{x, y} )
 =
 -\frac{3}{4}
 \norm{[ x,y]}^2_{\sigma_0}.
 $$
 \item 
 For all $z \in \Hb\setminus\set{0}$ and $x \in \Ha\setminus\set{0}$, the vectors $J_{Az}(Ax)$ and $z$ are orthogonal and
 $$ 
 K_{\sigma} 
 (\Pi_{J_{Az}(A x),z})
 =
 \frac{1}{4}
 \frac{ |A z|^4}{\norm{J_{Az}(A x)}^2_{\sigma_0} \norm{z}^2_{\sigma_0}}
 \norm{ x }_{\sigma_0}^2.
 $$
\end{enumerate}	
\end{lemma}
\begin{proof}
Due to Remark~\ref{rem:Bsigma0}, 
$
B_{\sigma}
( x , x)$, $B_{\sigma}( y , y )$, $B_{\sigma}( y , x )$,
$B_{\sigma}( x , y )$
all exist and are null. 
Thus, \eqref{def:K(X,Y)Arnold} 
immediately gives the claim (1).
The term $\alpha$ iof
\eqref{def:K(X,Y)Arnold} obviously vanishes, and again
Remark~\ref{rem:Bsigma0} gives
the existence and the vanishing of
$B_{J_{Az}(Ax)}$, $B_z$,
and $B_{\sigma_0}(J_{Az}(Ax),z)$
in the corresponding Arnold's formula
for the sectional curvature. 
From the property of the mapping $J_Z$,
$Z\in\Hb$, of a Hilbertian H-type group,
it is easy to notice that
$z$ and $J_{Az}(Ax)$ are $\sigma_0$-orthogonal.

Thus, by \eqref{def:K(X,Y)Arnold} applied to the $\sigma_0$-orthonormal basis $z/\|z\|_{\sigma_0}$ and $J_{Ay}(Ax)/\|J_{Ay}(Ax)\|_{\sigma_0}$, we get 
\begin{align*} 
K_{\sigma}
(\Pi_{J_{Az}(Ax),z})
&=\|\delta\|_{\sigma_0}^2=\frac14 
\norm{
	B_{\sigma_0}
	\left (
	\frac{z}{\norm{z}_{\sigma_0}}, 
	\frac{J_{Az}(Ax)}{\norm{J_{Az}(Ax)}}_{\sigma_0}
	\right ) 
}_{\sigma_0}^2  \\
&=\frac1{4\|z\|_{\sigma_0}^2\|J_{Az}(Ax)\|_{\sigma_0}^2}\norm{B_{\sigma_0}(z,J_{Az}(Ax))}^2_{\sigma_0} \\
&=\frac1{4\|z\|_{\sigma_0}^2\|J_{Az}(Ax)\|_{\sigma_0}^2}\norm{|Az|^2x}^2_{\sigma_0},
\end{align*}
where the last equality also relied on 
\eqref{eq:JAz} and immediately gives 
the claim (2).
\end{proof}

\begin{lemma}\label{esistenzavettoridegeneri2}
	Let $\sigma$ be a strictly weak graded Riemannian metric on a Hilbertian H-type group $\Ht = \Ha \oplus \Hb$.
	Then there exists $w_n \in A_\Ha(\Ha)$ such that 
	$|w_n|=1$
	and $\norm{w_n}_{\sigma_0} \to 0$ as $n\to+ \infty$.
\end{lemma}
\begin{proof}
We consider the sequence $h_n$ given by Lemma~\ref{esistenzavettoridegeneri}, 
hence $\|h_n\|_{\sigma_0}\to0$ and $|h_n|=1$.
The image $A_{\Ha}(\Ha)$ is dense in $\Ha$, as a consequence of Proposition~\ref{ImmagineOperatoreAssociatoMetrica}. Therefore, for each $n \in \N\setminus\{0\}$ we 
may choose $v_n \in A_{\Ha}(\Ha)$ such that $|v_n - h_n| \leq \frac{1}{2n}$, and therefore $|v_n| \to 1$. 
We define the unit vectors $w_n = \frac{v_n}{ |v_n|}$ and consider 
\[
\norm{v_n}_{\sigma_0} 
\leq 
\norm{h_n}_{\sigma_0} 
+ 
\norm{h_n - v_n}_{\sigma_0} 
\leq 
\norm{h_n}_{\sigma_0} 
+ 
c_0 |v_n - h_n| 
\leq
\norm{h_n}_{\sigma_0}  +\frac{c_0}{2n} 
\to 0,
\]
concluding the proof.
\end{proof}

\begin{theorem}\label{thm:unboundedcurvature}
Let $\sigma$ be a strictly weak, graded
Riemannian metric on a Hilbertian H-type group $\Ht = \Ha \oplus \Hb$. 
Then there exists a sequence $w_n\in \Ha$ such that for every $z \in \Hb\setminus\set{0}$ 
the following limits hold
\begin{equation}\label{eq:blowup}
	\lim_{n \to \infty}K_{\sigma}(\Pi_{w_n, J_zw_n})=-\infty
\quad\text{and}\quad \lim_{n \to \infty}K_{\sigma}(\Pi_{z,J_{Az}w_n})= +\infty.
\end{equation}
\end{theorem}

\begin{proof}
We consider the sequence $w_n\in A_\Ha(\Ha)\subset\Ha$ of Lemma~\ref{esistenzavettoridegeneri2} and define the vector
$$
\xi_n
= J_zw_n-\frac{w_n}{\norm{w_n}_{\sigma_0}}
\ban{ J_zw_n,\frac{w_n}{\norm{w_n}_{\sigma_0}} }_{\sigma_0}\in\Ha.
$$
By construction of $\xi_n$, the vectors 
$w_n/\norm{w_n}_{\sigma_0}$ and
$\xi_n/\norm{\xi_n}_{\sigma_0}$
are $\sigma_0$-orthonormal and span 
the 2-dimensional subspace $\Pi_{w_n,J_zw_n}$. 
By Lemma~\ref{Formulecurvaturapiani} and \eqref{eq:xJ_z}, we have 
$$
K_{\sigma}(\Pi_{w_n,J_zw_n})=-\frac{3}{4}\norm{\left[ \frac{w_n}{\norm{w_n}_{\sigma_0}}, \frac{\xi_n}{\norm{\xi_n}_{\sigma_0}} \right]}^2_{\sigma_0} 
=
-\frac{3}{4}
\frac{\norm{z}^2_{\sigma_0}}{\norm{w_n}^2_{\sigma_0}
	\norm{
		J_zw_n-\frac{w_n}{\norm{w_n}^2_{\sigma_0}}
		\ban{ J_z w_n,w_n }_{\sigma_0}
	}^2_{\sigma_0}
}.
$$
We consider the estimates 
\begin{align*}
\norm{
	J_zw_n-\frac{w_n}{\norm{w_n}^2_{\sigma_0}}
	\ban{ J_z w_n,w_n }_{\sigma_0}
}^2_{\sigma_0}
&\leq 2\,\pa{\|J_zw_n\|_{\sigma_0}^2+\frac{\ban{J_zw_n,w_n}_{\sigma_0}^2}{\|w_n\|_{\sigma_0}^2}} \\
&\le 4\, \|J_zw_n\|_{\sigma_0}^2\le 4c_0^2 |J_zw_n|^2 \\
&\le 4 c_0^2 |z|^2,
\end{align*}
where we have applied both \eqref{ContinuityRiemannianMetric} and \eqref{eq:J_znorm}.
It follows that
\[
\norm{\left[ \frac{w_n}{\norm{w_n}_{\sigma_0}}, \frac{\xi_n}{\norm{\xi_n}_{\sigma_0}} \right]}^2_{\sigma_0} 
\geq
\frac{1}{
	4c_0^2|z|^2
}
\frac{\norm{z}^2_{\sigma_0}}{\norm{w_n}^2_{\sigma_0}
}
\to+\infty,
\] 
proving the first limit of \eqref{eq:blowup}
To establish the second limit of \eqref{eq:blowup}, we consider the same previous sequence $w_n\in A_\Ha(\Ha)$,
along with $v_n \in \Ha$ such that $A v_n=w_n$.
By Lemma~\ref{Formulecurvaturapiani} we have
\begin{equation}
		K_{\sigma}(\Pi_{z,J_{Az}(A v_n)}) = 
		\frac{1}{4}
		\frac{
			|Az|^4}{\norm{J_{Az} (A v_n)}^2_{\sigma_0}\norm{z}^2_{\sigma_0}}
		\norm{v_n}_{\sigma_0}^2. 	
\end{equation} 
Again \eqref{ContinuityRiemannianMetric}
and \eqref{eq:J_znorm} give the inequalities
\[
\frac{1}{
	\norm{J_{Az}( A v_n)}_{\sigma_0} 
}
\geq 
\frac{1}{
	c_0 |J_{Az} (A v_n) |
}
=
\frac{1}{
	c_0 |Az| |A v_n|
}
=
\frac{1}{
	c_0 |Az|
}
> 0,
\]
where we have also use the condition 
$|w_n|=|Av_n|=1$. 
By Cauchy-Schwarz inequality, we get
$$
\norm{v_n}_{\sigma_0}
\geq
\ban{v_n ,\frac{w_n }{\norm{w_n}_{\sigma_0}} }_{\sigma_0}
=
\langle Av_n , w_n \rangle 
\frac{1}{\norm{w_n}_{\sigma_0}}
=
\frac{1}{\norm{w_n}_{\sigma_0}}
\to + \infty,
$$
that immediately yields
$K_{\sigma}(\Pi_{z,J_{Az}w_n}) \to + \infty$,
concluding the proof.
\end{proof}

{\bf Acknowledgements.} The authors wish to express their warm gratitude to Martin Bauer for providing valuable comments on the Michor–Mumford conjecture.

\bibliography{References}{}
\bibliographystyle{plain}

\end{document}